\documentclass[times,doublespace]{amsart}

\usepackage{graphicx}
\usepackage{amssymb}
\usepackage{epstopdf}

\usepackage[a4paper, total={6in, 9in}]{geometry}

\usepackage{amsmath,amsfonts,amsthm,mathrsfs,amssymb,cite}
\usepackage[usenames]{color}

\newtheorem{thm}{Theorem}[section]

\newtheorem{lem}{Lemma}[section]

\theoremstyle{definition}

\theoremstyle{remark}

\numberwithin{equation}{section}

\numberwithin{equation}{section}

\newcounter{saveeqn}

\newcommand{\eqnref}[1]{(\ref {#1})}

\newcommand{\Bz}{\mathbf{z}}

\newcommand{\Bx}{\mathbf{x}}
\newcommand{\By}{\mathbf{y}}

\newcommand{\Kcal}{\mathcal{K}}

\newcommand{\Scal}{\mathcal{S}}

\newcommand{\Ocal}{\mathcal{O}}

\newcommand{\ds}{\displaystyle}

\newcommand{\RR}{\mathbb{R}}

\newcommand{\p}{\partial}

\newcommand{\beq}{\begin{equation}}
\newcommand{\eeq}{\end{equation}}

\title[Asymptotic behavior of NP operator]{Asymptotic behavior of spectral of Neumann-Poincar\'e operator in Helmhotz system}

\author{Xiaoping Fang}
\address{School of Mathematics and Statistics, Hunan University of Commerce, Changsha 410205, China; Institute of Big Data and Internet Innovation, Hunan University of Commerce, Changsha  410205, China}
\email{fxp1222@163.com}

\author{Youjun Deng}
\address{School of Mathematics and Statistics, Central South University, Changsha 410083, China}
\email{youjundeng@csu.edu.cn, dengyijun\_001@163.com}

\date{} 
\begin{document}
\maketitle

\begin{abstract}
In this paper, we are concerned with the asymptotic behavior of the Neumann-Poincar\'e operator in Helmholtz system. By analyzing the asymptotic behavior of spherical Bessel function near the origin and/or approach higher order, we prove the asymptotic behavior of spectral of Neumann-Poincar\'e operator when frequency is small enough and/or the order is large enough. The results show that spectral of Neumann-Poincar\'e operator is continuous at the origin and converges to zero from the complex plane in general.



\end{abstract}

\section{Introduction}
Recently, spectral of Neumann-Poincar\'e operator have attracted much attention, for its applications in plasmon resonance \cite{ADM14,AMRZ17,AKa16,AKLi16,DLLi18,DLLi182,FDLi15,KLYu17}, cloaking due to anomalous localized resonance \cite{ACKL13,ACKL14,ACKLY13,AJKK18,CKKL14,LLi17,LLLW18} and enhancement of near cloaking \cite{HKLL13,AKLL1301,AKLLY13}. We also refer to \cite{HKLi17,KKLS16} for analysis of spectral of Neuamann-Poincar\'e operator in domains with corners. Most of the studies are based on the static (quasi-static) case, i.e., conductivity problem. In \cite{AKLi16,LLLW18,LLi17}, the authors consider the spectral of Neumann-Poincar\'e operator in Helmholtz system with finite frequency and use the result to analyze plasmon resonance and cloaking due to anomalous localized resonance phenomena. Mathematically, one consider the following Helmholtz system in $\RR^d$, $d=2, 3$:
\beq\label{eq:helm01}
\left\{
\begin{split}
&\nabla\cdot (\varepsilon(\Bx) \nabla u (\Bx)) + k^2 u(\Bx)= 0, \quad \quad \, \, \Bx \in \RR^d,\\
&\lim_{|\Bx|\rightarrow \infty}|\Bx|^{(d-1)/2}\Big(\frac{\Bx}{|\Bx|}\cdot\nabla u(\Bx)-i ku (\Bx)\Big)=0,
\end{split}
\right.
\eeq
where $\varepsilon(\Bx)$ denotes for the material parameter. Suppose $\varepsilon(\Bx)=(\varepsilon_0-1)\chi(D)+1$, where $\chi(D)$ is the indicator function for inclusion $D$ and $\varepsilon_0$ is the material parameter of $D$. The shape of the inclusion $D$ is essentially connected with the spectral of the Neumann-Poincar\'e operator $(\Kcal_D^k)^*$(see \eqnref{eq:layperpt2}). In static case, that is $k=0$ , the spectral of $(\Kcal_D^0)^*$ has been studied widely and the eigenvalues have been found elaborately for some special cases, i.e., $D$ is a disk, ball or ellipse. In \cite{LLi17}, the authors present one form of the eigenvalues of $(\Kcal_D^k)^*$ when $D$ is a ball for finite frequency $k$. However, the relation between asymptotic behavior of spectral of $(\Kcal_D^k)^*$ and spectral of $(\Kcal_D^0)^*$ is still not known. In this paper, we deal with this problem. We first derive some different forms of eigenvalues of $(\Kcal_D^k)^*$ when $D$ is a ball, and then show that the eigenvalues approach exactly to the eigenvalues of $(\Kcal_D^0)^*$ when $k$ goes to zero. We also show that the eigenvalues converges to complex number in general. Our main results for three dimensional case are listed in Theorem \ref{th:main01} and Theorem \ref{th:main02}. The asymptotic behavior of spectral of $(\Kcal_D^k)^*$, where $D$ is a disk, is listed in Theorem \ref{th:main03}.

\section{Layer potential and spectral of Neumann-Poincar\'e operator in $\RR^3$}
In this section, we present the spectral of Neumann-Poincar\'e operator. Before proceeding, we present some preliminary knowledge on layer potential techniques (cf. \cite{HK07:book, Ned}).

\subsection{Layer potentials}
Let $ G_k$ be the fundamental solution to the PDE operator $\Delta+k^2$ in $\RR^d$, $d=2, 3$, that is
\begin{equation}
\label{Gk} \ds G_k(\Bx) =
\left\{
\begin{split}
&-\frac{i}{4} H_0^{(1)}(k|\Bx|), \quad &d=2\\
&-\frac{e^{ik|\Bx|}}{4 \pi |\Bx|}, \quad &d=3
\end{split}
\right.
 \end{equation}
 where $H_0^{(1)}(k|\Bx|)$ is the Hankel function of first kind of order zero.
For any bounded Lipschitz domain $B\subset \RR^d$, $d=2, 3$, we denote by $\Scal_B^k: H^{-1/2}(\p B)\rightarrow H^{1}(\RR^d\setminus\p B)$ the single layer potential operator given by
\beq\label{eq:layperpt1}
\Scal_{B}^k[\phi](\Bx):=\int_{\p B}G_k(\Bx-\By)\phi(\By)\; d s_\By,
\eeq
and $(\Kcal_B^k)^*: H^{-1/2}(\p B)\rightarrow H^{-1/2}(\p B)$ the Neumann-Poincar\'e operator
\beq\label{eq:layperpt2}
(\Kcal_{B}^k)^{*}[\phi](\Bx):=\mbox{p.v.}\quad\int_{\p B}\frac{\p G_k(\Bx-\By)}{\p \nu}\phi(\By)\; d s_\By,
\eeq
where p.v. stands for the Cauchy principle value. In \eqref{eq:layperpt2} and also in what follows, unless otherwise specified, $\nu$ signifies the exterior unit normal vector to the boundary of the concerned domain.
It is known that the single layer potential operator $\Scal_B^k$ is continuous across $\p B$ and satisfies the following trace formula
\beq \label{eq:trace}
\frac{\p}{\p\nu}\Scal_B^k[\phi] \Big|_{\pm} = (\pm \frac{1}{2}I+
(\Kcal_{B}^k)^*)[\phi] \quad \mbox{on} \quad \p B, \eeq
where $\frac{\p }{\p \nu}$ stands for the normal derivative and the subscripts $\pm$ indicate the limits from outside and inside of a given inclusion $B$, respectively. In the following, if $k=0$, we formally set $G_k$ introduced in \eqref{Gk} to be $G_0$, and the other integral operators introduced above can also be formally defined when $k=0$.

\subsection{Spherical Bessel and Neumann functions}
In this part, we present some preliminary results for spherical Bessel and Neumann functions. Recall that the spherical Bessel and Neumann functions are solutions to the following spherical Bessel differential equation (see, e.g., \cite{CK}):
\beq\label{eq:besseleq}
t^2f''(t)+2tf'(t)+(t^2-n(n+1))f(t)=0, \quad n=0, 1, 2, \ldots.
\eeq
In the sequel, unless otherwise stated, $n$ is always chosen for nonnegative nature numbers, i.e., $n=0, 1, 2, \ldots$. The spherical Bessel and Neumann functions are then defined by
\beq\label{eq:besseldf}
j_n(t):=\sum_{l=0}^{\infty}\frac{(-1)^l}{2^l l!1\cdot3 \cdots (2n+2l+1)}t^{2l+n},
\eeq
and
\beq\label{eq:besseldf02}
y_n(t):=-\frac{(2n)!}{2^n n!}\sum_{l=0}^{\infty}\frac{(-1)^{l}2^{2l-n-1}}{2^ll!(-2n+1)(-2n+3)\cdots(-2n+2l-1)}t^{2l-n-1},
\eeq
The linear combination
\beq\label{eq:hankel}
h^{(1)}_n:=j_n(t)+i y_n(t)
\eeq
is called the spherical Hankel function of first kind of order $n$. For subsequent usage, we present the following famous Wronskian identity:
\beq\label{eq:Wronskian}
j_n'(t)h_n^{(1)}(t)-j_n(t){h_n^{(1)}}'(t)=-\frac{i}{t^2}.
\eeq
For $n$ sufficiently large enough, there holds the following asymptotic behavior (cf. \cite{CK}):
\beq\label{eq:asybes01}
j_n(t)=\frac{2^n n! t^n}{(2n+1)!}\Big(1+\Ocal\Big(\frac{1}{n}\Big)\Big),
\eeq
uniformly on compact subsets of $\RR$ and
\beq\label{eq:asynew01}
y_n(t)=\frac{(2n)!}{i2^n n!t^{n+1}}\Big(1+\Ocal\Big(\frac{1}{n}\Big)\Big),
\eeq
uniformly on compact subsets of $(0, \infty)$.
In what follows we define by $Y_n^m$ the spherical harmonics of order $n$ and degree $m$ and $P_n(t)$ the Legendre polynomial of order $n$.
For any $\Bx\in \RR^3$, let $\hat\Bx:=\Bx/|\Bx|$ be the unit vector.
We present the following Funk-Hecke formula (cf. \cite{Mu66}):
\begin{lem}\label{le:00}
Suppose that $f(t)$ is continuous for $t\in [-1, 1]$, then there holds that
\beq\label{eq:leFH01}
\int_{\mathbb{S}^2} f(\hat\Bx\cdot\hat\By) Y_n^m(\hat\By)ds= \lambda Y_n^m(\hat\Bx),
\eeq
with
\beq\label{eq:leFH02}
\lambda:=2\pi \int_{-1}^1 f(t)P_n(t)dt.
\eeq
\end{lem}


\subsection{Spectral of Neumann-Poincar\'e operator in $\RR^3$}
In this part, we shall present the spectral of Neumann-Poincar\'e operator $(\Kcal_{B_R}^{k})^*$, where $B_R$ is a ball with radius $R$. We have the following scaling result (see, e.g., \cite{ADM14}):
\begin{lem}\label{le:01}
The spectral of $(\Kcal_{B_R}^{k})^*$ is the same with the spectral of $(\Kcal_{B_1}^{kR})^*$.
\end{lem}
Without loss of generality, in the sequel, we only consider the spectral of $(\Kcal_B^k)^*$, where $B$ is a unit ball.
\begin{lem}\label{le:02}
There holds the following:
\beq\label{eq:NPspectral}
(\Kcal_B^k)^*[Y_n^m]=\Big(-\frac{1}{2}-ik^2j_n(k){h_n^{(1)}}'(k)\Big)Y_n^m=\Big(\frac{1}{2}-ik^2j_n'(k)h_n^{(1)}(k)\Big)Y_n^m.
\eeq
\end{lem}
We mention that \eqnref{eq:NPspectral} is proved in \cite{LLi17}, with some additional assumption. Here, we shall present a different proof without any additional assumption.
\begin{proof}
It is shown in \cite{CK} that
\beq\label{eq:lepf01}
\Scal_B^k[Y_n^m(\hat\Bz)](\hat\Bx)=-ikj_n(k)h_n^{(1)}(k|\Bx|)Y_n^m(\hat\Bx), \quad |\Bx|>1.
\eeq
By using the jump formula \eqnref{eq:trace} from the outside of $B$ one then has
\beq\label{eq:lepf02}
\Big(\frac{I}{2}+(\Kcal_B^k)^*\Big)[Y_n^m(\hat\Bz)](\hat\Bx)=-ik^2j_n(k){h_n^{(1)}}'(k)Y_n^m(\hat\Bx),
\eeq
which proves the first equality in \eqnref{eq:NPspectral}. By using \eqnref{eq:Wronskian} one thus has the second equality in \eqnref{eq:NPspectral}.

The proof is complete.
\end{proof}
We have another form of the spectral of Neumann-Poincar\'e operator. Before this, we present the following useful result (cf. \cite{LLi17}):
\begin{lem}\label{le:03}
For any $\phi\in H^{-1/2}(\p B)$, there holds the following identity:
\beq\label{eq:leax01}
(\Kcal_B^k)^*[\phi]=-\frac{1}{2}\Scal_B^k[\phi]+\frac{ik}{2}E_B^k[\phi],
\eeq
where the operator $E_B^k: H^{-1/2}(\p B)\rightarrow H^{1/2}(\p B)$ is defined by
\beq\label{eq:deEB}
E_B^k[\phi](\Bx):=-\frac{1}{4\pi}\int_{\p B} e^{ik|\Bx-\By|} \phi(\By)ds_\By.
\eeq
\end{lem}
\begin{lem}\label{le:04}
There holds the following:
\beq\label{eq:NPspectral02}
(\Kcal_B^k)^*[Y_n^m]=\frac{ik}{2}(j_n(k)h_n^{(1)}(k)+c_{n,k})Y_n^m,
\eeq
where $c_{n,k}$ is defined by
\beq\label{eq:leNP01}
c_{n,k}:=-\frac{1}{2}\int_{-1}^1 e^{ik\sqrt{2(1-t)}} P_n(t)dt.
\eeq
\end{lem}
\begin{proof}
By using \eqnref{eq:lepf01} and the continuous of $\Scal_B^k$ across $\p B$ one has
\beq\label{eq:lepf0101}
\Scal_B^k[Y_n^m]=-ikj_n(k)h_n^{(1)}(k)Y_n^m.
\eeq
Note that
$$
|\Bx-\By|=\sqrt{2-2\Bx\cdot\By}, \quad \Bx, \By\in \p B,
$$
together with the definition \eqnref{eq:deEB} and Funk-Hecke formula \eqnref{eq:leFH01}, one thus has
\beq\label{eq:lepf0102}
E_B^k[Y_n^m]=c_{n,k} Y_n^m.
\eeq
By substituting \eqnref{eq:lepf0101} and \eqnref{eq:lepf0102} back into \eqnref{eq:leax01} one thus has \eqnref{eq:NPspectral02}, which completes the proof.
\end{proof}
By combining the two forms of eigenvalues \eqnref{eq:NPspectral} and \eqnref{eq:NPspectral02}, one can find the following result:
\begin{lem}\label{le:0101}
There holds the following:
\beq\label{eq:le010101}
h_n^{(1)}(k)=\frac{1-ikc_{n,k}}{ik(j_n(k)+2kj_n'(k))},
\eeq
where $c_{n,k}$ is defined in \eqnref{eq:leNP01}. If $k=k_0$, where $j_n'(k_0)=-1/(2k_0)j_n(k_0)$ then the above equality is realized as an limit for $k\rightarrow k_0$.
\end{lem}

\section{Asymptotic behavior of spectral of Neumann-Poincar\'e operator}
In this section, we shall derive the asymptotic behavior of the Neumann-Poincar\'e operator $(\Kcal_B^k)^*$ when the frequency $k$ is sufficiently small or $n$ is sufficiently large enough. We first present some auxiliary results.
\begin{lem}\label{le:0102}
Let $c_{n, k}$ be defined in \eqnref{eq:leNP01}. Suppose $k$ is sufficiently small, then there holds the following:
\beq\label{eq:le010201}
c_{0,k}=-1-\frac{4}{3}ik+\Ocal(k^2),
\eeq
and
\beq\label{eq:le010202}
c_{n,k}=-\frac{\sqrt{2}}{2}ik\int_{-1}^1\sqrt{1-t}P_n(t)dt+\Ocal(k^2).
\eeq
\end{lem}
\begin{proof}
The proof is straight forward by using Taloy expansion and noticing the orthogonality of the Legendre polynomial $P_n(t)$.
\end{proof}
\begin{lem}\label{le:0103}
Suppose $k$ is sufficiently small, then there holds
\beq\label{eq:le010301}
kj_n'(k)=nj_n(k)+\Ocal(k^{n+2}).
\eeq
\end{lem}
\begin{proof}
We begin with $n=0$, then by $j_0(k)=\sin(k)/k$, one immediately has \eqnref{eq:le010301} by Taylor expansion.
Now suppose $n\neq 0$, then by using \eqnref{eq:besseldf} one has
$$
kj_n'(k)=n\frac{2^n n!}{(2n+1)!}k^n+ \Ocal(k^{n+2}), \quad j_n(k)=\frac{2^n n!}{(2n+1)!}k^n+ \Ocal(k^{n+2}),
$$
which verifies \eqnref{eq:le010301} and the proof is complete.
\end{proof}
We are now in the position of presenting our main results. Suppose that
\beq\label{eq:tau01}
(\Kcal_B^k)^*[Y_n^m]=\tau_{n, k} Y_n^m.
\eeq
One can find the explicit form of $\tau_{n, k}$ from \eqnref{eq:NPspectral} and \eqnref{eq:NPspectral02}.
\begin{thm}\label{th:main01}
Let $\tau_{n, k}$ be defined in \eqnref{eq:tau01}. Then there holds the following
\beq\label{eq:thmain01}
\lim_{k\rightarrow 0} \tau_{n, k}= \frac{1}{2(2n+1)}.
\eeq
\end{thm}
Before the proof of Theorem \ref{th:main01}, we want to make a remark. It is shown in \cite{HKLL13} that
\beq
(\Kcal_B^0)^*[Y_n^m]=\frac{1}{2(2n+1)}Y_n^m.
\eeq
Thus Theorem \ref{th:main01} indicates that the spectral of Neumann-Poincar\'e operator is continuous near the origin.
\begin{proof}
By using \eqnref{eq:NPspectral02} and \eqnref{eq:le010101} one has
\beq\label{eq:thpf01}
\tau_{n, k}=\frac{ik}{2}\Big(j_n(k)\frac{1-ikc_{n,k}}{ik(j_n(k)+2kj_n'(k))}+c_{n,k}\Big)
\eeq
If $n=0$, then by $j_0(k)=\sin(k)/k$, one can directly calculate \eqnref{eq:thpf01} with $n=0$ to get
\beq
\lim_{k\rightarrow 0}\tau_{0, k}=\lim_{k\rightarrow 0}\frac{i}{2}\frac{\sin(k)}{i\sin(k)+2i(k\cos(k)-\sin(k))}=\frac{1}{2}.
\eeq
If $n\neq 0$, by using Lemma \ref{le:0102} and Lemma \ref{le:0103}, one then has
\beq
\lim_{k\rightarrow 0}\tau_{n, k}=\lim_{k\rightarrow 0} \frac{1}{2}\frac{j_n(k)}{j_n(k)+2kj_n'(k)}=\lim_{k\rightarrow 0} \frac{1}{2}\frac{1}{1+2kj_n'(k)/j_n(k)}=\frac{1}{2(2n+1)}.
\eeq

The proof is complete.
\end{proof}

\begin{thm}\label{th:main02}
Let $\tau_{n, k}$ be defined in \eqnref{eq:tau01}. Then for $n$ sufficiently large enough, there holds the following:
\beq\label{eq:thmain01}
\tau_{n, k}= \frac{1}{2(2n+1)}+\frac{i}{2}kc_{n, k}+\Ocal\Big(\frac{k}{n^{3/2}}\Big),
\eeq
where
\beq\label{eq:behacnk01}
c_{n, k}=\Ocal\Big(\frac{1}{\sqrt{n}}\Big).
\eeq
\end{thm}
\begin{proof}
Suppose that $n$ is sufficiently large enough, then by using \eqnref{eq:asybes01}, one obtains that
\beq\label{eq:pfmain0101}
kj_n'(k)=\frac{n2^n n!k^n}{(2n+1)!}\Big(1+\Ocal\Big(\frac{1}{n}\Big)\Big).
\eeq
By using the orthogonality property of the Legendre polynomial $P_n(t)$, that is
\beq
\int_{-1}^1 P_m(t) P_n(t)=\frac{2}{2n+1}\delta_{mn},
\eeq
where $\delta_{mn}$ denotes for the Kronecker delta, which equals to one if $m=n$ and zero otherwise, and Cauchy-Schwarz inequality one derives that
$$
|c_{n, k}|=\Big|-\frac{1}{2}\int_{-1}^1 e^{ik\sqrt{2(1-t)}} P_n(t)dt\Big|\leq \frac{1}{2}\Big(\int_{-1}^1 \big| e^{ik\sqrt{2(1-t)}}\big|^2 dt\Big)^{1/2}\sqrt{\frac{2}{2n+1}}=\frac{1}{\sqrt{2n+1}}.
$$
Then by using \eqnref{eq:asybes01}, \eqnref{eq:thpf01} and \eqnref{eq:pfmain0101} one has
\beq\label{eq:pfmain0102}
\tau_{n, k}=\frac{1-ikc_{n,k}}{2(2n+1)}\Big(1+\Ocal\Big(\frac{1}{n}\Big)\Big)+\frac{ik}{2}c_{n,k}=\frac{1}{2(2n+1)}+\frac{i}{2}kc_{n, k}+\Ocal\Big(\frac{k}{n^{3/2}}\Big),
\eeq
which completes the proof.
\end{proof}
From \eqnref{eq:thmain01} one can find that the convergence rate of the spectral is no faster than $1/(4n)$. The coefficient $c_{n, k}$ also plays an important role for the convergence of the spectral and it is a complex number in general, thus the spectral of $(\Kcal_B^k)^*$ should converge to zero from the complex plane in general.
Let $\tau^{(0)}_{n, k}:=1/(2(2n+1))+\frac{i}{2}kc_{n, k}$ be the leading order in \eqnref{eq:thmain01},
the numerical illustrations of $\tau^{(0)}_{n, k}$, $0\leq k\leq 20$ with different choice of $n$ are presented in Figure \ref{fig1} and Figure \ref{fig2}. In Figure \ref{fig3} we illustrate the convergence of $|\tau^{(0)}_{n, k}|$ for different choice of $k$.


\begin{figure}
\begin{center}
  \includegraphics[width=2.5in,height=2.0in]{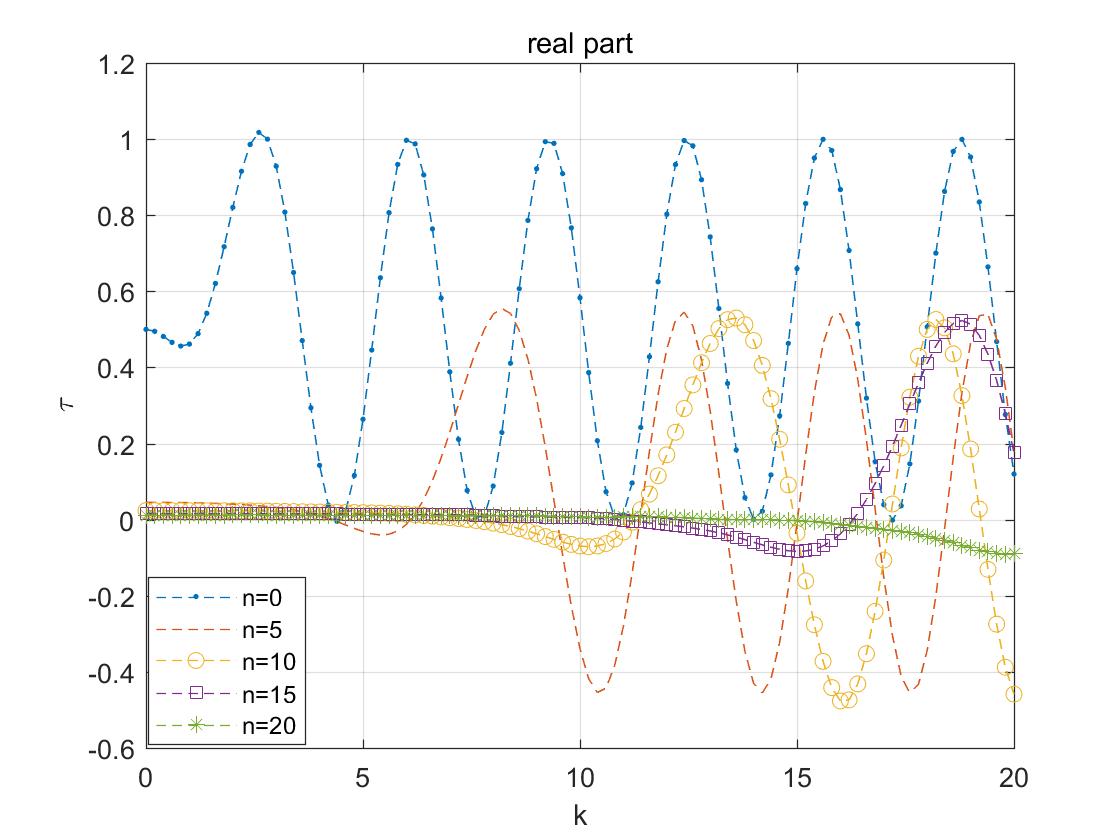}
  \includegraphics[width=2.5in,height=2.0in]{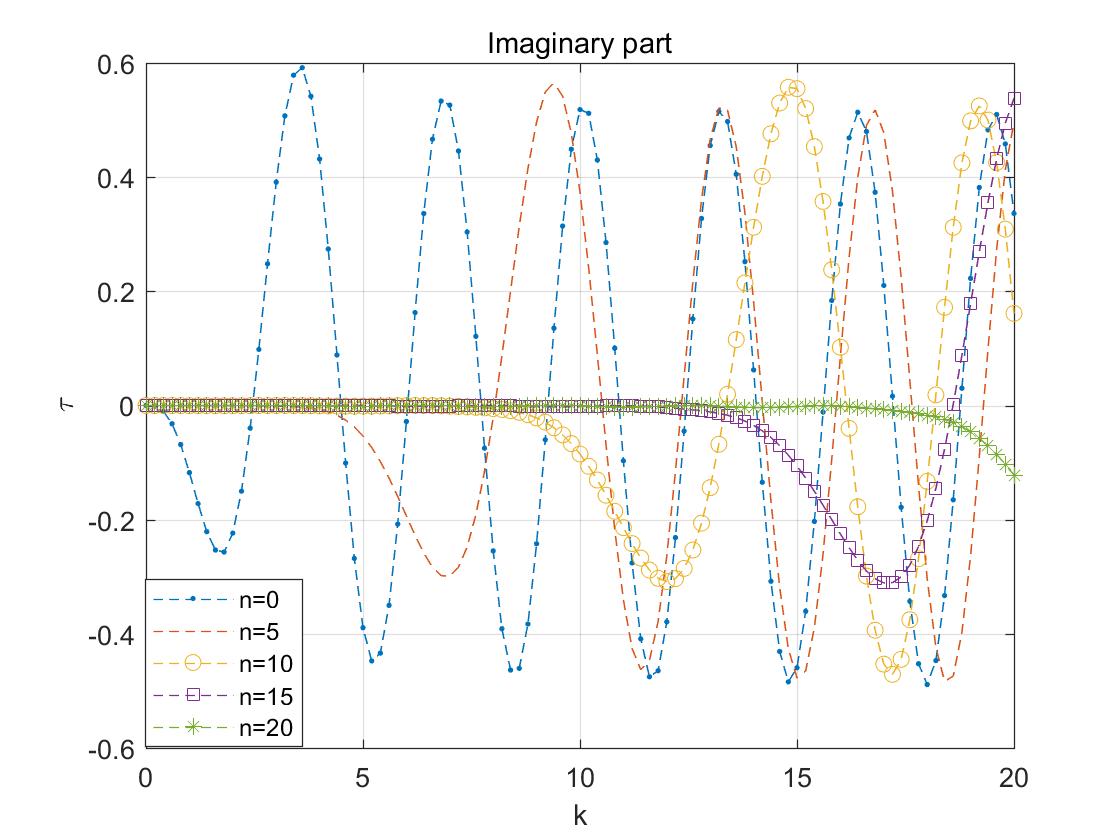}
  \includegraphics[width=2.5in,height=2.0in]{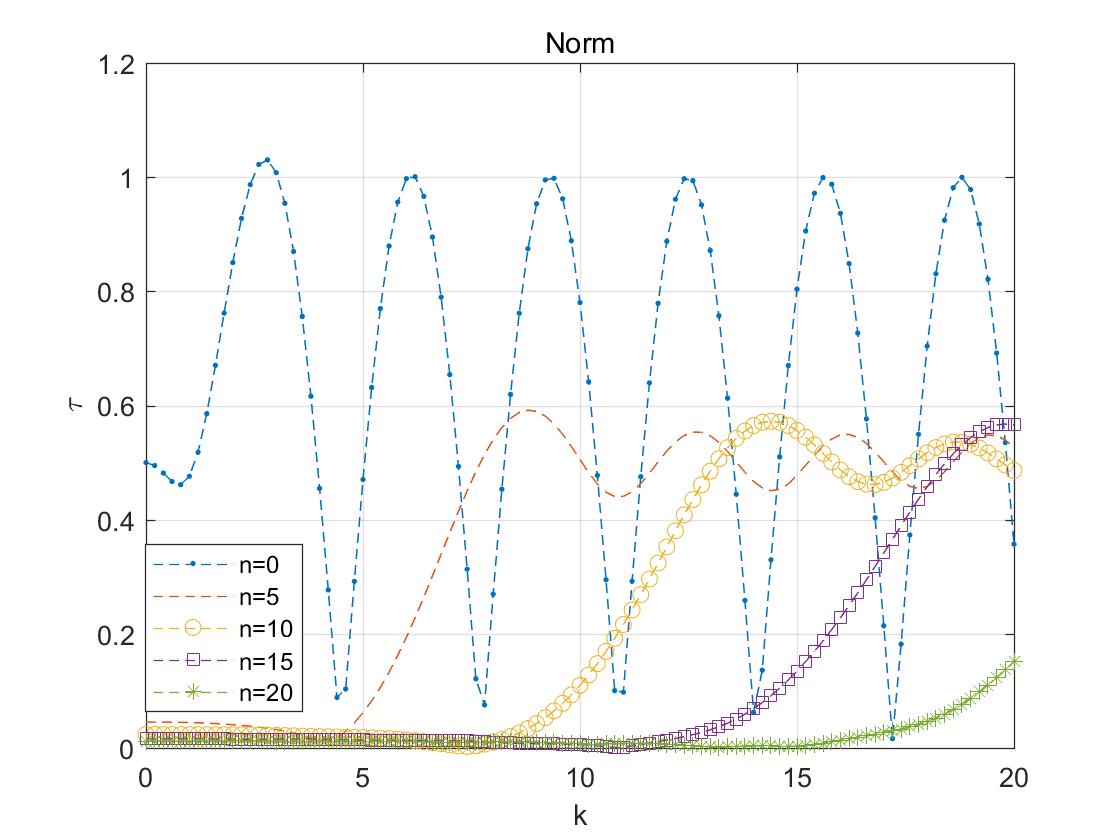}
  \end{center}
  \caption{Real part, imaginary part and modulus of $\tau^{(0)}_{n, k}$ with respect to $0\leq k\leq 20$. The curves in the figure represent for $n=0, 5, 10, 15, 20$, respectively. \label{fig1}}
\end{figure}

\begin{figure}
\begin{center}
  \includegraphics[width=2.5in,height=2.0in]{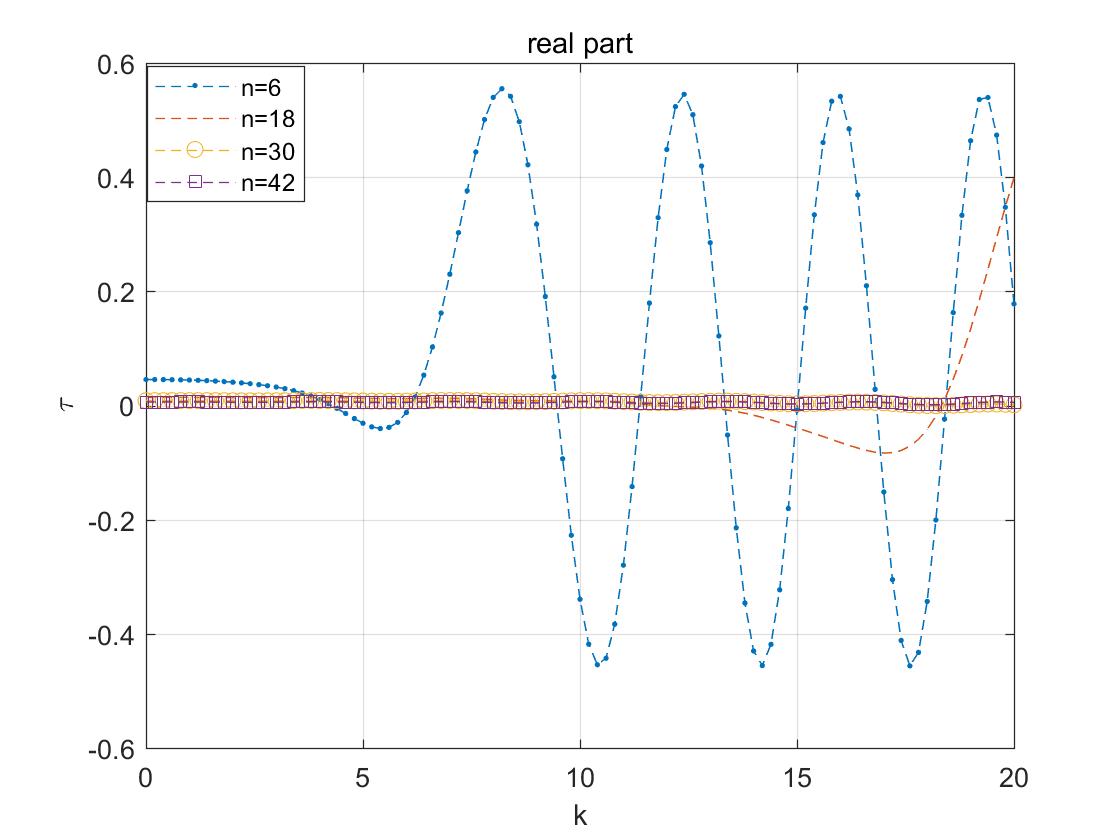}
  \includegraphics[width=2.5in,height=2.0in]{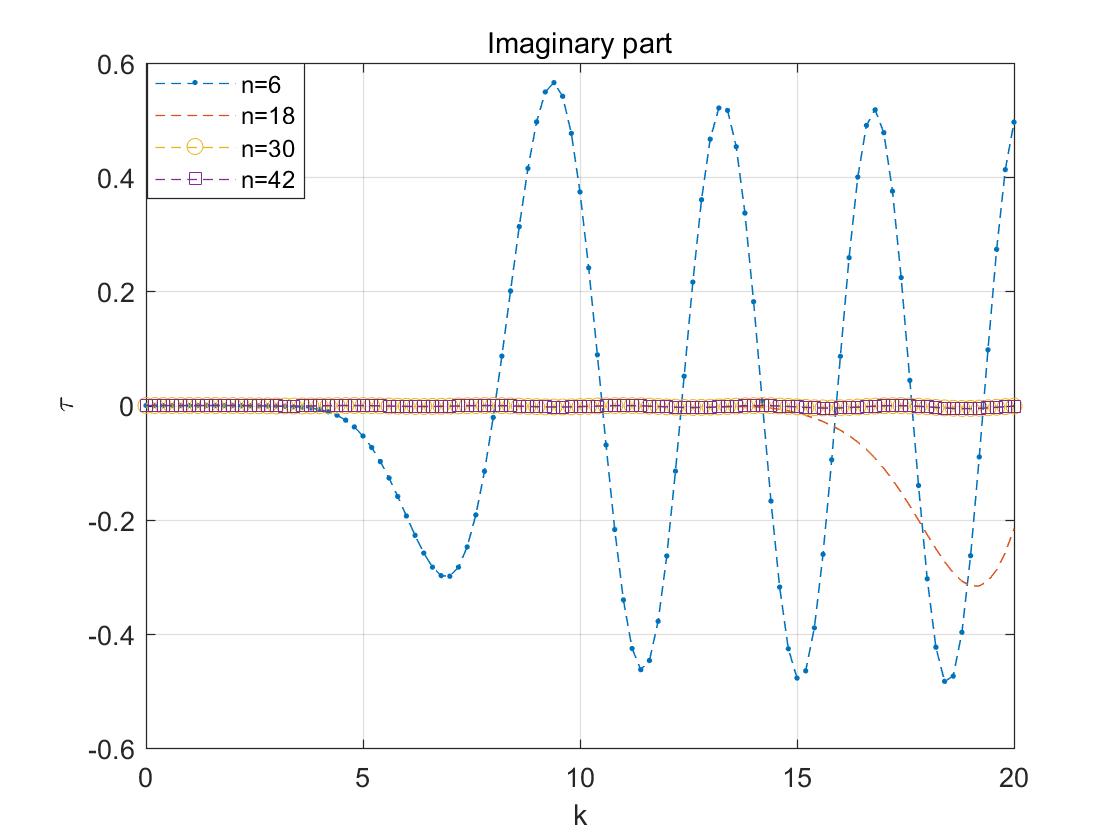}
  \includegraphics[width=2.5in,height=2.0in]{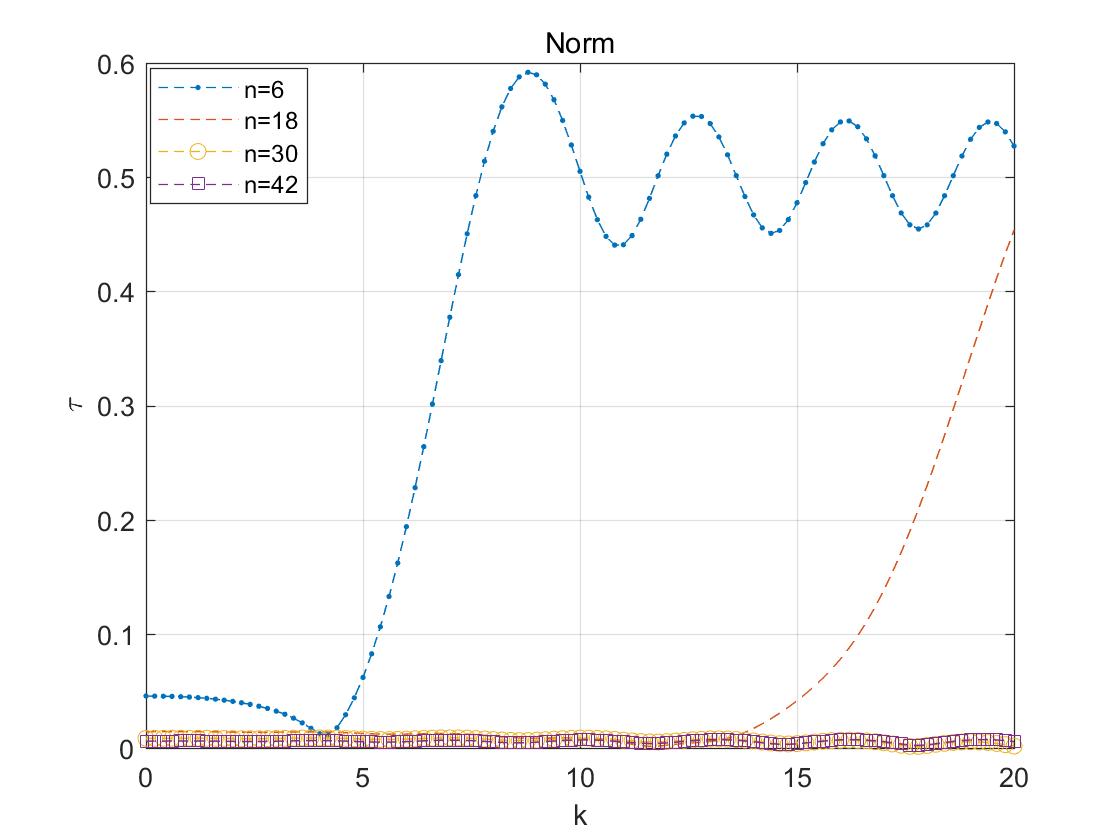}
  \end{center}
  \caption{Real part, imaginary part and modulus of $\tau^{(0)}_{n, k}$ with respect to $0\leq k\leq 20$. The curves in the figure represent for $n=6, 18, 30, 42$, respectively. \label{fig2}}
\end{figure}

\begin{figure}
\begin{center}
  \includegraphics[width=2.5in,height=2.0in]{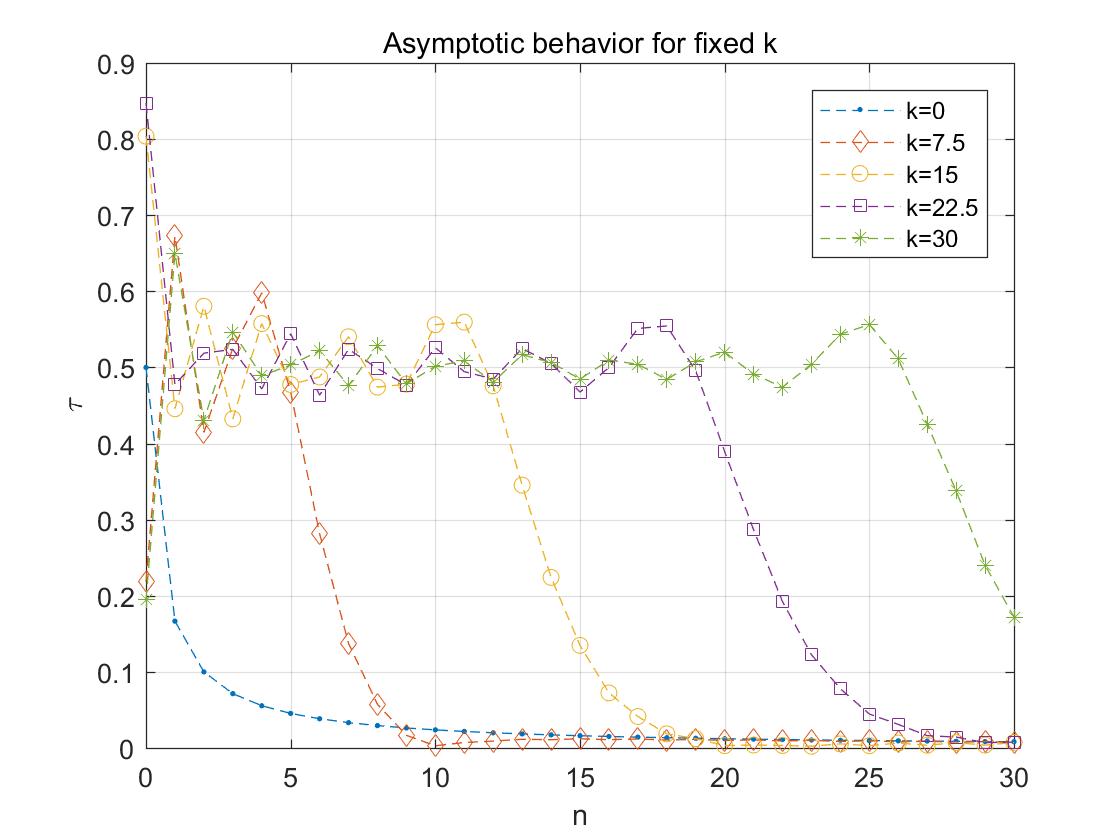}
  \end{center}
  \caption{Modulus of $\tau^{(0)}_{n, k}$ with respect to $0\leq n\leq 30$. The curves in the figure represent for $k=0, 7.5, 15, 22.5, 30$, respectively. \label{fig3}}
\end{figure}

\section{Two dimensional case}
\subsection{Spectral of Neumann-Poincar\'e operator in $\RR^2$}
In this section, we shall present the spectral of Neumann-Poincare\'e operator in $\RR^2$ for Helmholtz system \eqnref{eq:helm01}. We suppose that $Q$ is a disk of radius one and $\Bx:=(\|\Bx\|\cos\theta_x, \|\Bx\|\sin\theta_x)$. Recall the Graf's formula (\cite{ATWa14,AbSt64})
\beq\label{eq:sp2d01}
H_0^{(1)}(k|\Bx-\By|)=\sum_{n\in\mathbb{Z}}H_n^{(1)}(k\|\Bx\|)e^{in\theta_x}J_n(k\|\By\|)e^{-in\theta_y}, \quad \|\Bx\|>\|\By\|,
\eeq
where $H_n^{(1)}$ is the Hankel function of the first kind of order $n$, that is
\beq\label{eq:sp2d02}
H_n^{(1)}(t)= J_n(t)+ iN_n(t),
\eeq
where $J_n$ and $N_n$ are Bessel functions of order $n$ of the first and the second kind, respectively. Similar to the three dimensional case, we have the following result:
\begin{lem}
There holds the following for $\|\Bx\|>1$
\beq\label{eq:sp2dle01}
\Scal_Q^k[e^{in\theta}](\Bx)=-\frac{i\pi}{2}J_n(k)H_n^{(1)}(k\|\Bx\|)e^{in\theta_x}.
\eeq
\end{lem}
\begin{proof}
By using \eqnref{eq:sp2d01} and the orthogonality of $e^{in\theta}$ on a unit circle, $n\in \mathbb{Z}$, one has
\beq
\begin{split}
\Scal_Q^k[e^{in\theta}](\Bx)=&-\frac{i}{4}\int_{\p Q} \sum_{m\in\mathbb{Z}}H_m^{(1)}(k\|\Bx\|)e^{im\theta_x}J_m(k\|\By\|)e^{-im\theta_y} e^{in\theta_y}ds_\By \\
=&-\frac{i}{4}\int_0^{2\pi}\sum_{m\in\mathbb{Z}}H_m^{(1)}(k\|\Bx\|)e^{im\theta_x}J_m(k)e^{-i(m-n)\theta_y} d\theta_y \\
=&-\frac{i\pi}{2}J_n(k)H_n^{(1)}(k\|\Bx\|)e^{in\theta_x},
\end{split}
\eeq
which completes the proof.
\end{proof}
In what follows, we present the Wronskian identity for $J_n$ and $H_n$, which is
\beq\label{eq:Wronskian01}
J_n'(t)H_n^{(1)}(t)-J_n(t){H_n^{(1)}}'(t)=-\frac{2i}{\pi t}.
\eeq
By using \eqnref{eq:sp2dle01} and \eqnref{eq:Wronskian01} one thus has the eigenvalues of $(\Kcal_Q^k)^*$ in the following:
\begin{lem}
There holds the following
\beq\label{eq:NPspectral01}
(\Kcal_Q^k)^*[e^{in\theta}]=\Big(-\frac{1}{2}-\frac{i\pi}{2}kJ_n(k){H_n^{(1)}}'(k)\Big)e^{in\theta}
=\Big(\frac{1}{2}-\frac{i\pi}{2}kJ_n'(k)H_n^{(1)}(k)\Big)e^{in\theta}.
\eeq
\end{lem}
\begin{proof}
The proof follows similar to the proof of Lemma \ref{le:02}.
\end{proof}

\subsection{Asymptotic behavior}
In the last part, we have present the spectral of $(\Kcal_Q^k)^*$ when $Q$ is a unit disk in \eqnref{eq:NPspectral01}. We shall show the asymptotic behavior of the spectral when $k$ is sufficiently small, or $n$ is sufficiently large. Recall that $J_n(k)$ and $H_n^{(1)}(k)$ admits the following asymptotic behavior:
\beq\label{eq:asynp2d01}
J_n(k)=\left\{
\begin{split}
&1- \frac{k^2}{4}+ \Ocal(k^4), & n=0 \quad \mbox{and} \quad k<<1 \\
&\frac{1}{\Gamma(n+1)}\Big(\frac{k}{2}\Big)^n\Big(1-\frac{1}{n+1}\Big(\frac{k}{2}\Big)^2\\
&+\frac{1}{2(n+1)(n+2)}\Big(\frac{k}{2}\Big)^4
+\Ocal\Big(\frac{k^6}{(n+1)^3}\Big)\Big),  & n\geq 1 \quad k<<\sqrt{n+1}
\end{split}
\right.
\eeq
and 
\beq\label{eq:asynp2d02}
H_n^{(1)}(k)=\left\{
\begin{split}
&1+i\frac{2}{\pi}\Big(\ln\frac{k}{2}+\gamma\Big)+\Ocal(k), & n=0 \quad \mbox{and} \quad k<<1 \\
&-i\frac{\Gamma(n)}{\pi}\Big(\frac{2}{k}\Big)^n
\Big(1+\frac{1}{n-1}\Big(\frac{k}{2}\Big)^2\\
&+\frac{1}{2(n-1)(n-2)}\Big(\frac{k}{2}\Big)^4+\Ocal\Big(\frac{k^6}{n^3}\Big)\Big),  & n\geq 1 \quad k<<\sqrt{n+1}
\end{split}
\right.
\eeq
where $\Gamma$ is the Gamma function and $\gamma=0.5772...$ is the Euler-Mascheroni constant. We now present the asymptotic behavior of $(K_Q^k)^*$ in two dimensional case as follows:
\begin{thm}\label{th:main03}
Suppose
\beq
(\Kcal_Q^k)^*[e^{in\theta}]=\kappa_{n, k} e^{in\theta}.
\eeq
Then there holds the following:
\beq\label{eq:thmain0301}
\lim_{k\rightarrow 0}\kappa_{n, k}= 
\left\{
\begin{split}
&\frac{1}{2}, & n=0 \\
& 0. & n\geq 1
\end{split}
\right.
\eeq
and for $n$ sufficiently large enough there holds
\beq\label{eq:thmain0302}
\kappa_{n, k}=-\frac{k^2}{4n(n-1)(n+1)}+\Ocal\Big(\frac{k^4}{n^4}\Big).
\eeq
\end{thm}
\begin{proof}
Suppose $n=0$, and $k$ is sufficiently small, then by using the first part in \eqnref{eq:asynp2d01} and \eqnref{eq:asynp2d02} one has
\beq\label{eq:thmain0303}
J_0'(k)=-\frac{k}{2}+\Ocal(k^3), \quad \mbox{and} \quad H_0^{(1)}(k)=i\frac{2}{\pi}\ln\frac{k}{2}+ \Ocal(1).
\eeq
By substituting \eqnref{eq:thmain0303} into the second identity in \eqnref{eq:NPspectral01}, there holds
$$
\kappa_{0,k}=\frac{1}{2}-i\frac{\pi}{2}k\big(-\frac{k}{2}+\Ocal(k^3)\big)\big(i\frac{2}{\pi}\ln\frac{k}{2}+ \Ocal(1)\big)=\frac{1}{2}-\frac{k^2}{2}\ln\frac{k}{2}+\Ocal(k^2),
$$
which proves the first case in \eqnref{eq:thmain0301}. Now fix $k$ and suppose $n$ is sufficiently large, from \eqnref{eq:NPspectral01} one has that
\beq\label{eq:thmain0304}
(J_n(k)H_n^{(1)}(k))'=J_n'(k)H_n^{(1)}(k)+J_n(k){H_n^{(1)}}'(k)=i\frac{4}{k\pi}\kappa_{n,k}.
\eeq
By using \eqnref{eq:asynp2d01} and \eqnref{eq:asynp2d02} one obtains
\beq\label{eq:thmain0305}
\begin{split}
J_n(k)H_n^{(1)}(k)=&-i\frac{1}{n\pi}\Big(1+\Big(\frac{1}{n-1}-\frac{1}{n+1}\Big)\Big(\frac{k}{2}\Big)^{2}+\\
&\quad\Big(\frac{1}{2(n+1)(n+2)}+\frac{1}{2(n-1)(n-2)}-\frac{1}{(n-1)(n+1)}\Big)\Big(\frac{k}{2}\Big)^4+\Ocal\Big(\frac{k^6}{n^3}\Big)\Big)\\
=&-i\frac{1}{n\pi}\Big(1+\frac{2}{(n-1)(n+1)}\Big(\frac{k}{2}\Big)^{2}+\Ocal\Big(\frac{k^4}{n^3}\Big)\Big).
\end{split}
\eeq
Together with \eqnref{eq:thmain0304}, one finally obtains
\beq\label{eq:thmain0306}
\begin{split}
\kappa_{n,k}=&-\frac{k\pi}{4}\frac{1}{n\pi}\Big(1+\frac{2}{(n-1)(n+1)}\Big(\frac{k}{2}\Big)^{2}+\Ocal\Big(\frac{k^4}{n^3}\Big)\Big)'\\
=&-\frac{k^2}{4n(n-1)(n+1)}+\Ocal\Big(\frac{k^4}{n^4}\Big),
\end{split}
\eeq
which completes the proof.
\end{proof}
Theorem \ref{th:main03} shows that the convergence rate of $\kappa_{n,k}$ is of order $1/n^3$ as $k$ fixed and $n$ goes to infinity. The numerical illustrations are listed in Figure \ref{fig4} and Figure \ref{fig5}. We mention that the degenerate rate of the eigenvalues of Neumann-Poincar\'e operator in two dimensional case is quite faster than than in three dimensional case, especially when $k$ is small. We refer to \cite{AKM18} for deriving of exponential decay rate of Neumann-Poincar\'e operator in static case.

\begin{figure}
\begin{center}
  \includegraphics[width=2.5in,height=2.0in]{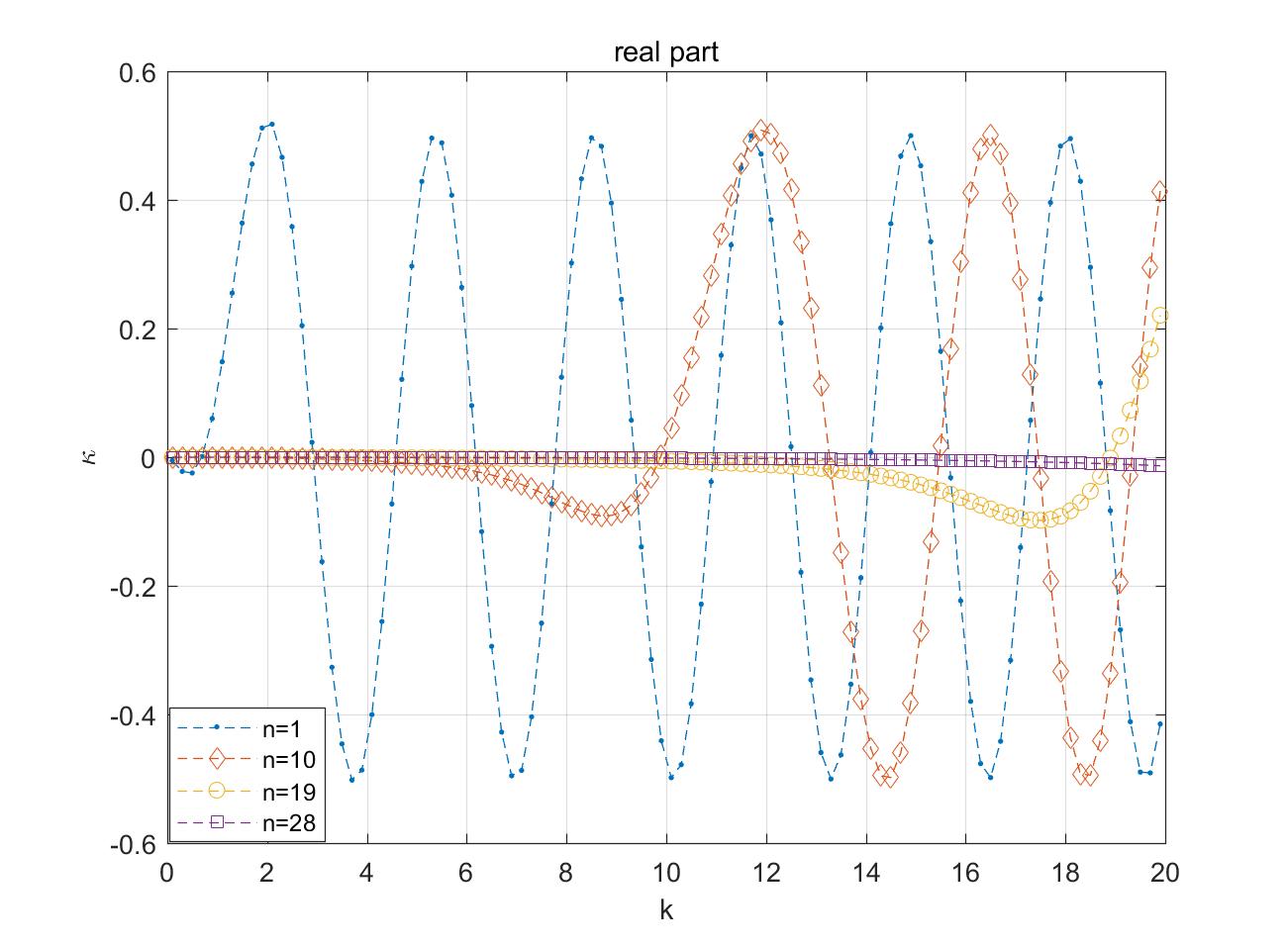}
  \includegraphics[width=2.5in,height=2.0in]{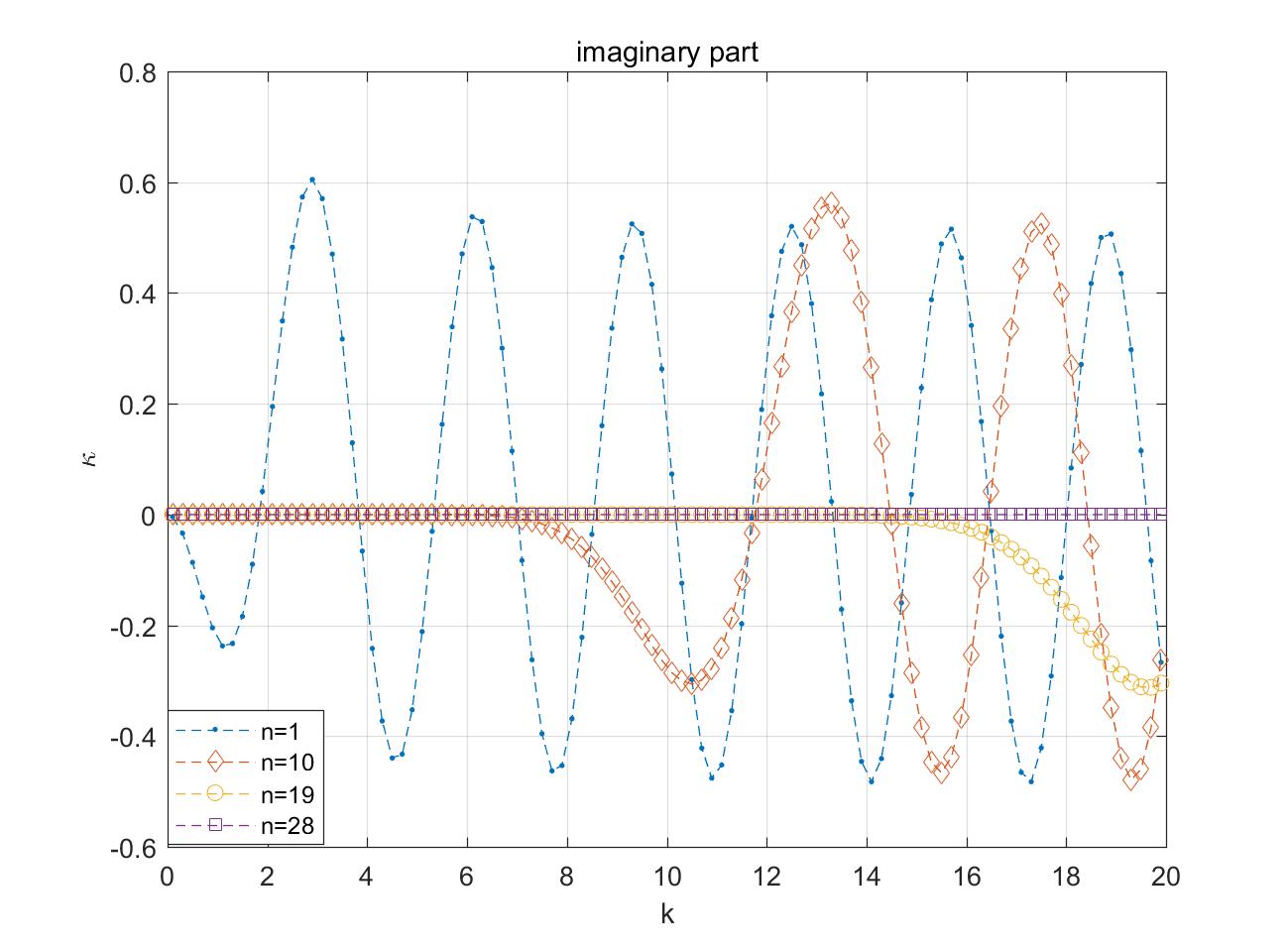}
  \includegraphics[width=2.5in,height=2.0in]{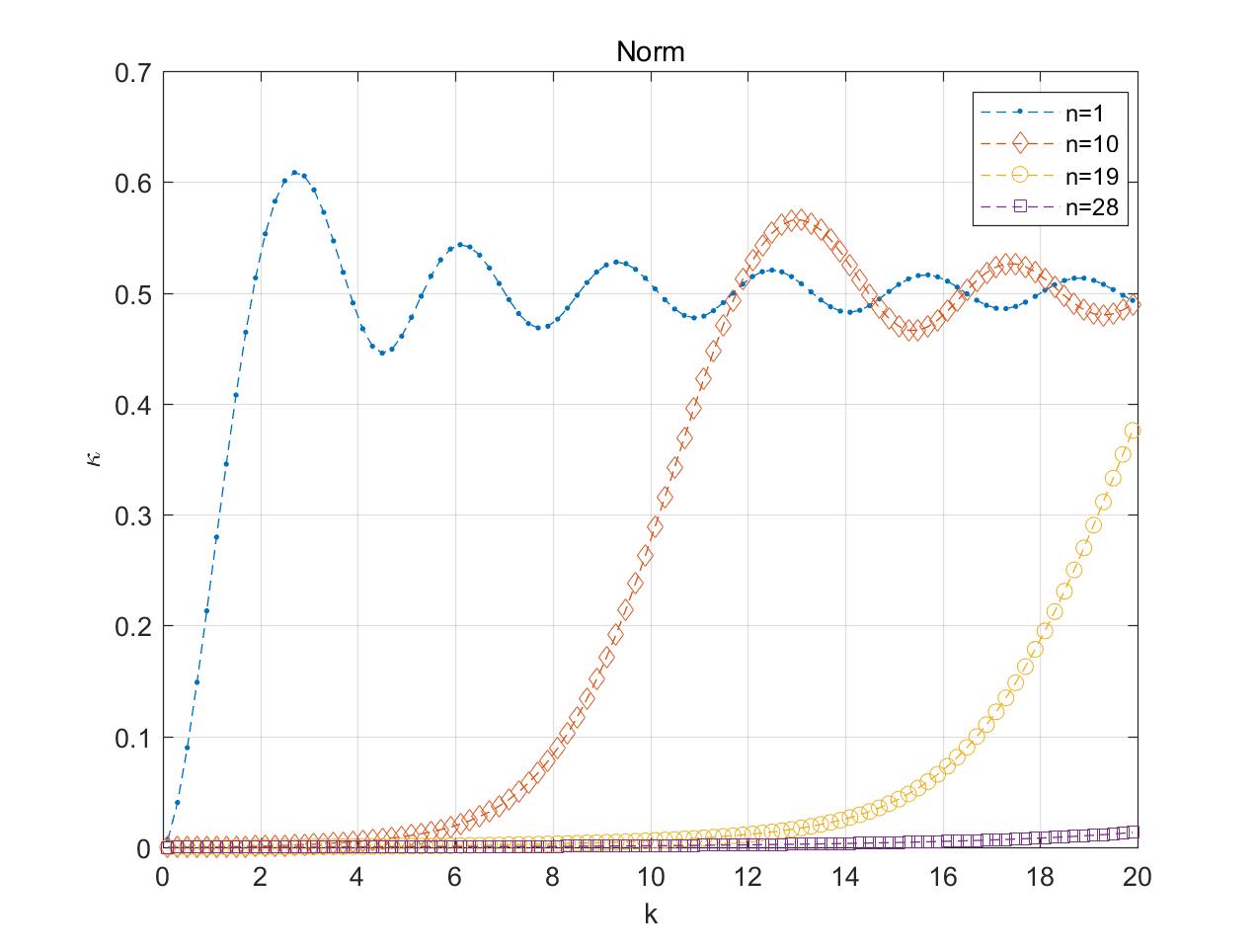}
  \end{center}
  \caption{Real part, imaginary part and modulus of $\kappa_{n, k}$ with respect to $0\leq k\leq 20$. The curves in the figure represent for $n=1, 10, 19, 28$, respectively. \label{fig4}}
\end{figure}

\begin{figure}
\begin{center}
  \includegraphics[width=2.5in,height=2.0in]{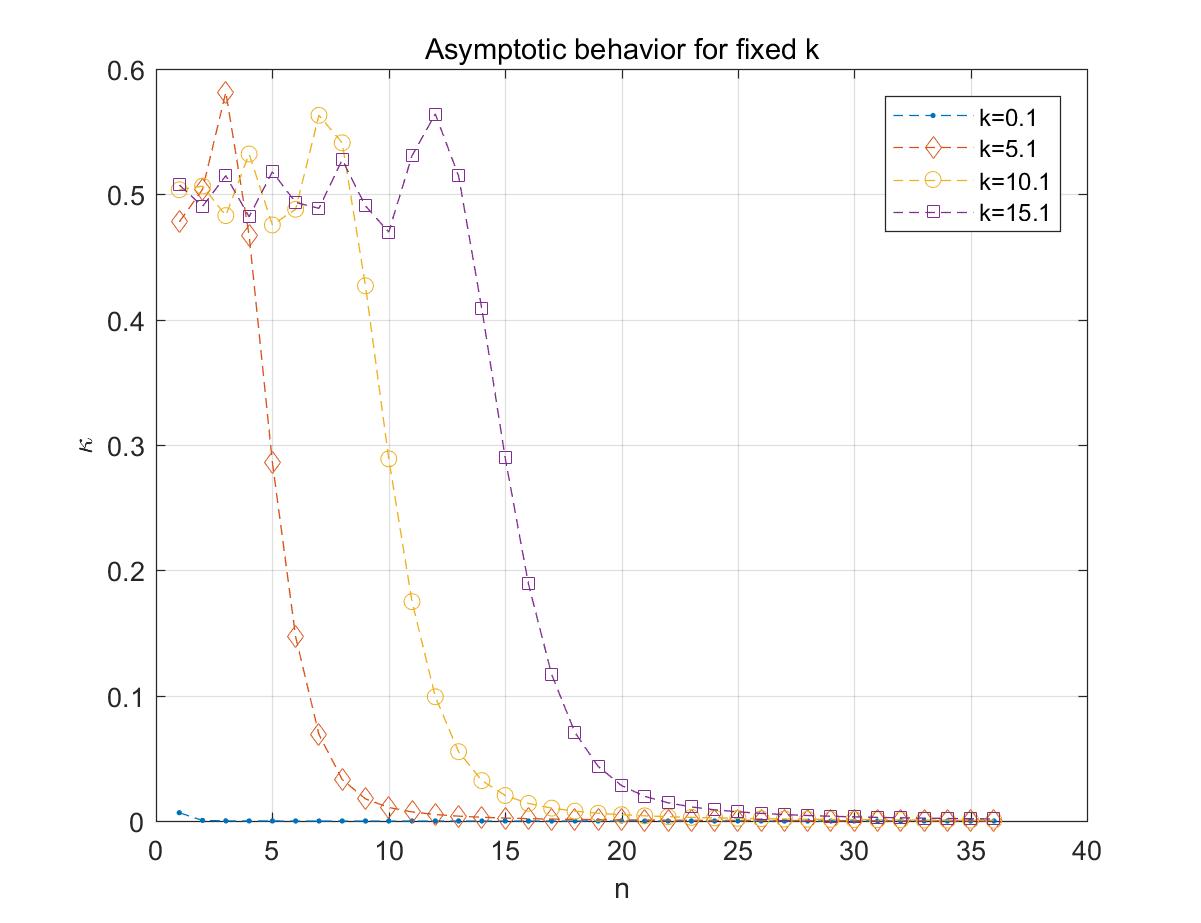}
  \end{center}
  \caption{Asymptotic behavior of $\kappa_{n, k}$ with respect to $0\leq n\leq 36$. The curves in the figure represent for $k=0.1, 5.1, 10.1, 15.1$, respectively.   \label{fig5}}
\end{figure}

\section{Concluding Remark}
In this paper, we have shown the asymptotic behavior of the spectral of Neumann-Poincar\'e operator in Helmholtz system in two aspects. First, if the frequency turns to zero, then the spectral is exactly the spectral of Neumann-Poncar\'e in static system, i.e., the spectral is continuous near the origin. Second, we show how the behavior of eigenvalues at infinity. 
\section*{Acknowledgment}
The authors wish to thank Professor Hongyu Liu for valuable discussions.
The work of Youjun Deng was supported by NSF grant of China No. 11601528, NSF grant of Hunan No. 2018JJ3622, Innovation-Driven Project of Central South University No. 2018CX041. The work of Xiaoping Fang was supported by NSF grant of Hunan No. 2017JJ3432, Grant HNJNPJ200805, 2018 Hunan statistical research project No. 11 and Major Project for National Natural Science Foundation of China No. 71790615.

\end{document}